\documentclass[11pt,a4paper,reqno]{amsart}
\usepackage{epsfig,amssymb,latexsym,amscd}

\usepackage[all,2cell]{xy}
\xyoption{v2}
\UseAllTwocells
\theoremstyle{plain}

\newtheorem{teo}{Theorem}[section]
\newtheorem{theo}[teo]{Theorem}
\newtheorem{coro}[teo]{Corollary}
\newtheorem{lema}[teo]{Lemma}
\newtheorem{prop}[teo]{Proposition}

\theoremstyle{remark}

\newtheorem{rk}[teo]{Remark}

\newtheorem{example}[teo]{Example}

\theoremstyle{definition}

\newtheorem{defi}[teo]{Definition}

\newcommand{\N}{\mathbb{N}}

\begin{document}
\sloppy

\title[The Igusa-Todorov function for comodules]{The Igusa Todorov function for comodules}

\subjclass[2000]{16} 

\begin{abstract}
We define the Igusa-Todorov function in the context of finite dimensional comodules and prove that a coalgebra is left qcF if and only if it is left semiperfect and its Igusa-Todorov function on each right finite dimensional comodule is zero.
\end{abstract}

\author{Mariana Haim, Marcelo Lanzilotta and Gustavo Mata}
\address{Universidad de la Rep\'ublica\\
Montevideo\\Uruguay\\}
\email{negra@cmat.edu.uy, marclan@fing.edu.uy, gmata@fing.edu.uy}
\thanks{The authors want to thank ANII-FCE 2007-059}

\maketitle

\section{Introduction}
\noindent
The Igusa-Todorov function (IT-function) appeared first in \cite{kn:igusatodorov} and has been considered again in \cite{kn:marchuard} and \cite{kn:hlm}. It is a new homological tool that generalises the notion of injective dimension (see Lemma \ref{gen}). In \cite{kn:marchuard}, the authors proved that, for artinian rings, the selfinjectivity can be characterised by the nullity of the IT- function on each finitely generated module. In other words, they proved that a ring is quasi-Frobenius if and only if it is right artinian and its IT-function is zero on each finitely generated right module.\\
\noindent
A coalgebra $C$ is said to be left (right) quasi-co-Frobenius (qcF) if every indecomposable injective right (left) $C$-comodule is projective. Since indecomposable projective comodules are finite dimensional, a left (right) qcF coalgebra is left (right) semiperfect, that is, all indecomposable injective right (left) comodules are finite dimensional. \\
We define here an IT-function in the context of finite dimensional comodules and we prove that a coalgebra $C$ is left qcF if and only if it is left semiperfect and its IT-function is zero. \\
Even if this equivalence can be seen as a possible dual version of
the problem treated in \cite{kn:marchuard}, it is not strictly the
case (we deal with coalgebras while the authors in
\cite{kn:marchuard} deal with artinian rings) and the proof given
here uses quite different ideas and tools.

\noindent
It is worth to remark that while the notion of quasi-Frobenius ring is left-right symmetric, the notion of qcF coalgebra is not. This applied to a coalgebra that is left and also right semiperfect will give an example of a coalgebra for which its IT-function of right comodules is zero, while its IT-function for left comodules is not (see Example \ref{example}). \\
\noindent
In what follows we present a brief description of the contents of this paper.\\
In Section \ref{pre} we recall a few basic notations and give the definitions, main properties and examples needed to understand and treat the question.\\
In Section \ref{mainresult} we prove the main result mentioned above and some needed previous lemmas.\\
In Section \ref{last} we use some computations with the IT-function, to deduce some well known facts about qcF coalgebras.

\section{The Igusa-Todorov function}\label{pre}
\subsection{Some notations}
In this work, $C$ will be a coalgebra over a field $\Bbbk$ and we will denote by $\mathcal {M}^C$ and $^C\!\!\mathcal {M}$ the categories of right and left comodules over $C$ respectively and by $\mathcal {M}^C_f$ and $^C\!\!\mathcal {M}_f$ the respective complete subcategories of finite dimensional comodules. Since $\mathcal {M}^C$ and $^C\!\!\mathcal {M}$ are Grothendieck categories, every object in them has an injective envelope (see for example \cite{kn:rumanos}).

\subsection{The Igusa-Todorov function on comodules}
Let $C$ be a coalgebra and $K(C)$ be the free abelian group
generated by all symbols $[M]$ with $M \in \ \mathcal {M}^C$ under
the relations

\begin{enumerate}
  \item $[A]-[B]-[C]$ if $A\cong B \oplus C$,
  \item $[I]$ if $I$ is injective.
\end{enumerate}

\noindent Then $K(C)$ is the free abelian group generated by all
isomorphism classes of indecomposable non injective objects in
$\mathcal {M}^C$. As the syzygy $\Omega^{-1}$  respects direct sums
and sends injective comodules to 0, it gives rise to a group
morphism (that we also call $\Omega^{-1}$) $\Omega^{-1}: K(C)
\rightarrow K(C)$.

\vspace{2mm}

\noindent For any $M \in \ \mathcal {M}^C_f$, let $\langle M \rangle$ denote the subgroup of $K(C)$ generated by all the symbols $[N]$, where $N$ is an indecomposable non injective direct summand of $M$. Since the rank of $\Omega^{-1}(\langle M \rangle)$ is less or equal to the rank of $\langle M \rangle$, which is finite, it exists a non-negative integer $n$ such that the rank of $\Omega^{-n}(\langle M\rangle)$ is equal to the rank of $\Omega^{-i}(\langle M \rangle)$ for all $i \geq n$. Let $\varphi(M)$ denote the least such $n\in \mathbb N$.

\vspace{2mm}
\noindent
The main properties of $\varphi$ are summarised in the following lemma, whose version for Artin algebras has been proved, almost all in \cite{kn:igusatodorov} and the last in \cite{kn:hlm} and can be easily adapted to obtain the version for coalgebras.

\begin{lema}\label{gen}(\cite{kn:igusatodorov}, \cite{kn:hlm})
  Let $C$ be a coalgebra and $M$, $N \in \ \mathcal {M}^C_f$.
  \begin{enumerate}
    \item If the injective dimension of $M$, $id M$, is finite, then $id M = \varphi(M)$.
    \item If $M$ is indecomposible of infinite injective dimension, then $\varphi(M) = 0$.
    \item $\varphi(N\oplus M) \geq \varphi (M)$.
    \item $\varphi (M^{k}) = \varphi (M)$ if $k\geq 1$.
    \item $\varphi (M) \leq \varphi(\Omega^{-1} M) + 1$, whenever $\Omega^{-1} M$ is finite dimensional.
  \end{enumerate}
\end{lema}

\noindent
In a similar way it is possible to define $\varphi$ on the category $^C\!\!\mathcal {M}_f$. We will use the same notation for both functions, when no confusion arises.

\begin{rk}
Note that $\varphi (M)=0$ mains that rk$\Omega^{-n}(\langle M \rangle)$ remains constant for all $n\in \N$.
\end{rk}

\begin{defi}
  For a coalgebra $C$ we define
  $$ dim_\varphi(^C\!\! \mathcal M_f) = sup\left \{\varphi(M) \mbox{ with }M \in\ ^C\!\!\mathcal {M}_f \right \},$$
  $$dim_\varphi(\mathcal M_f^C) = sup\left \{ \varphi(M) \mbox{ with }M \in \mathcal {M}^{C}_f \right \}. $$
\end{defi}
\noindent
\ \\
The following example shows that $dim_\varphi (^C\!\!\mathcal M_f)$
and $dim_\varphi(\mathcal M^C_f)$ can be different.

\begin{example}\label{example}
Consider the quiver
$$
Q=\xymatrix
 {\cdots \ar[r] & \stackrel{4}{\cdot} \ar[r] & \stackrel{3}{\cdot} \ar[r] & \stackrel{2}{\cdot} \ar[r]  & \stackrel{1}{\cdot} \ar[r] & \stackrel{0}{\cdot} \\ }
 $$
\noindent and let $C$ be the coalgebra whose elements are all paths
in $\Bbbk Q$ of length less or equal to one. Each comodule $M$ in $\
^{C}\mathcal {M}_f$ can be seen as a  $\Bbbk Q$-representation
$(M_i, T_i)_{i\in \N}$ where $T_i: M_{i+1}\rightarrow M_i$ is such
that $T_i.T_{i+1}=0$, for all $i\in \N$. It is easy to check that
every such representation can be decomposed into a sum of
indecomposable representations of the form (see \cite{kn:simson}):

\begin{itemize}
  \item $ \xymatrix
{\cdots \ar[r]^{0} & 0 \ar[r]^{0} & 0 \ar[r]^{0} & \Bbbk \ar[r]^{1_{\Bbbk}} & \Bbbk \ar[r]^{0} & \cdots \ar[r]^{0} & 0 \ar[r]^{0} & 0\\}$

  \item $ \xymatrix
{\cdots \ar[r]^{0} & 0 \ar[r]^{0} & 0 \ar[r]^{0} & \Bbbk \ar[r]^{0} & 0 \ar[r]^{0} & \cdots \ar[r]^{0} & 0 \ar[r]^{0} & 0 \\ }$
\end{itemize}
\noindent
Note that representations of the first type are injective, while representations of the second type are simple. As $\varphi(M\oplus I)=\varphi(M)$, whenever $I$ is injective, in order to prove that $\varphi(M)=0$, for every comodule $M$ in $\ ^{C}\mathcal {M}_f$, it is enough to show it for cosemisimple comodules. If we consider:

$$M = \xymatrix
{   \cdots \ar[r]^{0} &  0 \ar[r]^{0} & 0 \ar[r]^{0} & V_k \ar[r]^{0} & \cdots \ar[r]^{0} & V_1 \ar[r]^{0} & V_0 \\ },$$
\noindent
after applying $\Omega^{-1}$ to $M$ we obtain the representation

$$\Omega^{-1}(M) = \xymatrix
{   \cdots \ar[r]^{0} & 0 \ar[r]^{0} & V_k \ar[r]^{0} & \cdots \ar[r]^{0} & V_1 \ar[r]^{0} & V_0 \ar[r] & 0 .\\ }$$

\noindent
Hence the ranks of $\langle\Omega^{-1}(M)\rangle$ and $ \langle M \rangle$ are equal, and then, by induction (because $\Omega^{-1}(M)$ is cosemisimple), the ranks of $\langle\Omega^{-n}(M)\rangle$ and $ \langle M \rangle$ are equal. Then $\varphi(M) = 0$ for every cosemisimple object in $^{C}\mathcal {M}_f$, so $dim_\varphi(\ ^C\mathcal M_f) = 0$.

\vspace{2mm}
\noindent
On the other hand, right $C$-comodules can be seen as $Q'$-representations $(M_i, T_i)_{i\in \N}$ where $T_i: M_{i}\rightarrow M_{i+1}$ is such that $T_i.T_{i+1}=0$, for all $i\in \N$ and

$$ Q'=\xymatrix
{ \cdots & \stackrel{4}{\cdot} \ar[l] & \stackrel{3}{\cdot} \ar[l] & \stackrel{2}{\cdot} \ar[l]  & \stackrel{1}{\cdot} \ar[l] & \stackrel{0}{\cdot} \ar[l] \\ } $$

\noindent Now, the right $C$-comodule:

$$M_n = \xymatrix
{   \cdots  & 0 \ar[l]^{0} & 0 \ar[l]^{0} & \Bbbk \ar[l]^{0} & \ar[l]^{0}  \cdots & 0 \ar[l]^{0} &  0 \ar[l]^{0} \\ }$$
\noindent (where the non zero vector space $\Bbbk$ is placed in the vertex $n$), has injective dimension $n$, so $\varphi(M_n) = n$ and therefore $dim_\varphi (\mathcal M^C_f)=\infty$. In particular $dim_\varphi(\ ^C\mathcal M_f) \neq dim_\varphi(\mathcal M^C_f)$.
\end{example}
\noindent
To finish this section, we recall the notions of quasi-co-Frobenius and semiperfect.
\begin{defi}\label{semip}
A coalgebra $C$ is said to be
\begin{itemize}
\item {\em left quasi-co-Frobenius}, shortly {\em left  qcF}, if every injective right $C$-comodule is projective.
\item {\em right quasi-co-Frobenius}, shortly {\em right qcF}, if $C^{op}$ is left qcF.
\item {\em left semiperfect} if all injective envelopes of simple right $C$-comodules are finite dimensional. Equivalently, if the category of left $C$-comodules has enough projectives (see \cite{kn:rumanos}).
\item {\em right semiperfect} if $C^{op}$ is left semiperfect.
\end{itemize}
\end{defi}

\section{The main result}\label{mainresult}

\noindent
In this section we will prove the main result of this work, stating that a coalgebra $C$ is left $qcF$ if and only if it is left semiperfect and verifies $dim_\varphi(\mathcal M^C_f)=0$ (Theorem \ref{main}).\\
\noindent
In order to do this, we start by proving some auxiliary results. In concrete, and Lemmas \ref{top}, \ref{epi} and \ref{aux} will be needed in the proof of Theorem \ref{main}.

\begin{lema}\label{top}
Let $C$ be a left semiperfect coalgebra with $dim_\varphi(\mathcal M^C_f)=0$. For any right simple $C$-comodule, $Top (E(S))$ is (defined and) simple.
\end{lema}
\begin{proof}
\noindent
As $E(S)$ is finite dimensional, its top is (defined and) cosemisimple. We will prove that it is simple.\\
Suppose there are two simple right $C$-comodules $S_1, S_2$ such that $S_1\oplus S_2$ is a direct summand of $Top(E(S))$, and consider the following short exact sequences:
$$\xymatrix
{\sigma_1:   & 0 \ar[r] & K_1 \ar[r] & E(S) \ar[r]  &S_1 \ar[r] & 0\\
\sigma_2:   & 0 \ar[r] & K_2 \ar[r] & E(S) \ar[r]  &S_2 \ar[r] & 0\\
\sigma_3:   & 0 \ar[r] & K_3 \ar[r] & E(S) \ar[r]  &S_1 \bigoplus S_2 \ar[r] & 0\\ }$$
\noindent
Note that $E(S)$ has simple socle, and $K_i$ is a non injective indecomposable comodule (since its socle is simple), for $i\in \{1,2,3\}$. Moreover, $K_i$ is non zero, for $i\in \{1,2,3\}$, since $S_1$ and $S_2$ are non zero. Also, it is clear that $K_i \not \cong K_3$, for $i\in \{1,2\}$. So $rk \langle \{ [K_1], [K_2], [K_3]\}\rangle \geq 2$. Note also that $E(S)$ is the injective envelope of $K_i$, for $i\in \{1,2,3\}$.\\
\noindent If $S_1 \cong S_2$, then $rk \langle \{ [S_1], [S_2], [S_1 \bigoplus S_2]\}\rangle=1$, then $\phi (K_1\bigoplus K_2 \bigoplus K_3) \geq 1$, contradicting $dim_\varphi(\mathcal M^C_f)=0$. If not, then $rk \langle \{ [S_1], [S_2], [S_1 \bigoplus S_2]\}\rangle=2$ but also $K_1 \not \cong K_2$, so $rk \langle \{ [K_1], [K_2], [K_3]\}\rangle=3$, arising again to a contradiction.
\end{proof}
\noindent
The following is a technical result.
\begin{lema}\label{epi}
 Let $f:\bigoplus_{i\in I} U_i \to V$ be an epimorphism of right $C$-comodules, where $V$ has simple top. Then for some $i\in I$, $f_i: U_i \to V$ defined by restricting $f$ to $U_i$ is surjective.
\end{lema}

\begin{proof}
 Since $V$ has simple top it has a unique maximal subcomodule $M$. If, for all $i\in I$, we have that $f_i$ is not surjective, then $Im(f_i) \subseteq M$, so $Im(f) =\sum_{i\in I} Im(f_i) \subseteq M$, contradicting that $f$ is surjective.
\end{proof}

\begin{lema}\label{aux}
If $P$ is a projective object in the category of finite dimensional
right comodules, then it is projective as a right comodule.
\end{lema}
\begin{proof}
As $P$ is finite dimensional, its dual $P^*$ is an injective object
in the category of finite dimensional left comodules. From Theorem
2.4.17 in \cite{kn:rumanos}, we deduce that $P^*$ is an injective
left comodule and therefore $P$ is projective.
\end{proof}

\begin{theo}\label{main}
A coalgebra $C$ is left qcF if and only if it is left semiperfect and $dim_\varphi(\mathcal M^C_f)=0$.
\end{theo}

\begin{proof}

Assume first that $C$ is left qcF. It is well known that then $C$ is left semiperfect (see \cite{kn:rumanos}). In order to prove that $dim_\varphi(\mathcal M^C_f)=0$, note that as $C$ is left $qcF$ every injective right $C$-comodule is projective, so the group morphism $\Omega^{-1}: K(C) \rightarrow K(C)$ is injective. Indeed, the injective envelope of any finite dimensional non injective right comodule $M$ will be the projective cover of $\Omega^{-1}(M)$, so $\Omega \circ \Omega^{-1}=id_{K(C)}$. This implies that $\Omega^{-1}$ preserve ranks of subgroups, so $dim_\varphi(\mathcal M^C_f)=0$.\\
\ \\
\noindent
Suppose now that $C$ is left semiperfect and that $dim_\varphi(\mathcal M^C_f)=0$. We will prove that any injective right $C$-comodule is projective. \\
\noindent
Let $E$ be the injective envelope of a simple right $C$-comodule $S$. Assume that $E$ is not projective and consider the following commutative diagram:
$$\xymatrix
{\sigma_1: &0 \ar[r]  & X \ar[d]^\iota \ar[r]& U \ar[r] \ar[d] & E \ar[r] \ar@{=}[d]& 0\\
\sigma_2: & 0 \ar[r] & E(X) \ar[r] & E' \ar[r]  &E \ar[r] & 0 }, $$
where
\begin{itemize}
\item $\sigma_1$ is a non zero short exact sequence (it does exist since $E$ is not projective), with
\begin{itemize}
\item $X$ finite dimensional (see Lemma \ref{aux}),
\item $U$ indecomposable (note that $U$ is finite dimensional -since $X$ and $E$ are- and then we can apply Lemma \ref{epi} and assume is indecomposable).
\end{itemize}
\item $(E(X), \iota)$ is the injective envelope of $X$ and $\sigma_2=\iota . \sigma_1$ the pushout.
\end{itemize}
\noindent
Note that $E'$ is injective (since $E(X)$ and $E$ are) and that $X$ is not injective (since $\sigma_1 \neq 0$). Moreover, if we consider the following commutative diagram:
$$ \xymatrix
{ & & 0 \ar[d] & 0\ar[d] \\
\sigma_1: & 0 \ar[r] &X \ar[r] \ar[d]^f& U \ar[r] \ar[d]& E \ar@{=}[d] \ar[r]&0\\
\sigma_2: & 0 \ar[r] &E(X) \ar[r] \ar[d]^f& E' \ar[r] \ar[d] & E \ar[r]&0\\
& & \Omega^{-1} (X) \ar[d] &\Omega^{-1} (U)\bigoplus E'' \ar[d]\\
& & 0 &0 }$$
we get from the snake lemma that $\Omega^{-1} (X) \cong \Omega^{-1} (U)\bigoplus E''$. So $[\Omega^{-1} (X)]=[\Omega^{-1} (U)\bigoplus E'']=[\Omega^{-1} (U)]+[E'']=[\Omega^{-1} (U)]$. Since $E\neq 0$, then no summand of $X$ is isomorphic to $U$. Therefore rk$\Omega^{-1}\langle X\oplus U\rangle \leq$ rk$\langle X\oplus U\rangle - 1$ and then $dim_\varphi(\mathcal M^C_f)\geq 1$, a contradiction. Then $E$ is a projective module.\\
\end{proof}

\begin{section}{Some consequences}\label{last}
\noindent
The characterization of qcF coalgebras given by Theorem \ref{main} allows us to prove some known results about them, by using the tools given by the IT-function.

\begin{prop}\label{final}\ \\
\noindent
If $C$ is a left $qcF$ coalgebra, then:
\begin{enumerate}
\item[(a)] Every indecomposable injective right $C$-comodule has simple top.
\item[(b)] Every indecomposable projective left $C$-comodule has simple socle.
\end{enumerate}
\end{prop}
\begin{proof}
\begin{enumerate}
\item[(a)] It is an immediate consequence of Lemma \ref{top}, after applying Theorem \ref{main}.
\item[(b)] It is known that every indecomposable projective left $C$-comodule $P$ is finite dimensional (see \cite{gt}) and then $P^*$ is an indecomposable injective right $C$-comodule (see for example \cite{kn:rumanos}, Chapter 2). But then note that $Top(P^*)=\left (Soc P\right )^*$. As $Top(P^*)$ is simple (and then finite dimensional) $Soc(P)$ also is simple.
\end{enumerate}
\end{proof}

\begin{prop}\label{nu}
Let $C$ be a left qcF coalgebra. Let $\mathcal S_r$ be a set of representatives of the isomorphism classes of simple right $C$-comodules. We can define a function
$$\ \nu_r: \mathcal S_r \to \mathcal S_r, \ \nu_r (S)= Top (E(S)) $$
that turns out to be injective.\\
\end{prop}
\begin{proof}
\noindent
By Theorem \ref{main}, we know that $C$ is left semiperfect and that $dim_\varphi (\mathcal M ^C_f)=0$. From Lemma \ref{top} we can define $\ \nu$, that will be proved injective. Suppose it is not. Then, there are two simple non isomorphic right $C$-comodules $S_1$ and $S_2$ such that $T_1=Top (E(S_1))$ and $T_2=Top (E(S_2))$ are isomorphic. So we have the following short exact sequences:
$$\xymatrix
{
\sigma_1:   & 0 \ar[r] & J_1 \ar[r] & E(S_1) \ar[r]  &T_1 \ar[r] & 0\\
\sigma_2:   & 0 \ar[r] & J_2 \ar[r] & E(S_2) \ar[r]  &T_2 \ar[r] & 0
},$$
with $soc(J_1)=S_1$ and $soc(J_2)=S_2$. Now, $J_1$ and $J_2$ are non injective indecomposable comodules (since they have simple socle and $\sigma_1$, $\sigma_2$ are not zero).\\
\noindent
As $T_1\cong T_2$ and $dim_\varphi(\mathcal M^C_f)=0$, we get $J_1 \cong J_2$ and therefore their socles are isomorphic: $S_1\cong S_2$.
\end{proof}

\noindent
The function $\nu$ appears in contexts of categories of modules with finitely many simple modules (where it is bijective) and is usually called the {\em Nakayama permutation} (see for example \cite{kn:kur}).
\ \\
\noindent
Note that every cosemisimple coalgebra $C$ verifies $dim_\varphi (\ ^C\mathcal M_f)=dim_\varphi(\mathcal M^C_f)=0$, since every (right or left) comodule is injective. Next proposition deals with the non cosemisimple case.

\begin{prop}\label{simpleinjective}
Let $C$ be a left qcF coalgebra. If $C$ is indecomposable not simple, there are no simple injective right $C$-comodules.
\end{prop}
\begin{proof}
Let $S$ be a simple injective right $C$-comodule.  As $C$ is not simple, there is some simple right comodule $T\neq S$ (otherwise $C=E(S)=S$) with $Ext^1_C(S,T)\neq 0$ (see \cite{kn:chin}). \\
Let $\sigma$ be a non splitting exact sequence from $T$ to $S$ and consider the following commutative diagram:
$$\xymatrix
{\sigma: & 0 \ar[r]  & T \ar@{=}[d] \ar[r]& M \ar[r] \ar[d] &S \ar[r] \ar[d]^f& 0\\
        & 0 \ar[r] & T \ar[r] & E(T) \ar[r]  &X \ar[r] & 0}$$
\noindent
As $\sigma$ is non zero, we have that $f$ is non zero and then injective (since $S$ is simple). Moreover, as $E(T)$ has simple top (by Proposition \ref{final}), we get that $X$ is indecomposable. Thus, $S\cong X$, since $S$ is injective. We conclude that the injective dimension of $T$ is $1$, contradicting the assumption that $dim_\varphi(\mathcal M^C_f)=0$.\\
\end{proof}
\noindent

\begin{coro}
If $C$ is a left and right $qcF$ coalgebra, then $\nu_r$ is bijective.
\end{coro}
\begin{proof}
As $C$ is left and right $qcF$, we have that $\nu_r:S_r\to S_r$ and $\nu_l:S_l\to S_l$ (defined similarly) are injective. Now define:
    $$    \mu_r: S_r \to S_r, \ \ \ \mu_r(S)=soc(P(S)) ,$$
where $P(S)$ is the projective cover of $S$ (note that as $C$ is right semiperfect $\mathcal M^C_f$ has enough projectives).\\
\noindent
From the diagram
$$\xymatrix
{ 0 \ar[r]  & \mbox { Soc}(P(S)) \ar[d] \ar[r]^{\iota} & E(\mbox {Soc}(P(S))) \ar[r] \ar[d] & \mbox{coker}(\iota) \ar[r] \ar[d]& 0\\
 0 \ar[r] & \mbox{ ker}(\pi) \ar[r] & P(S) \ar[r]^{\pi}  & S \ar[r] & 0}$$
\noindent it is clear that $\nu_r$ and $\mu_r$ are mutually inverses.
\end{proof}
\ \\

\noindent
It is known that a coalgebra is left (right) qcF if and only if it is left (right) semiperfect and it generates the category of its left comodules (see \cite{kn:blas}). After Theorem \ref{main}, this result can be restated as follows:

\begin{coro}\label{fin}
Let $C$ be a left semiperfect coalgebra. $C$ generates the category of its left comodules if and only if $dim_\varphi(\mathcal M^C_f)=0$.
\end{coro}
\noindent
The question of whether the equivalence of Corollary \ref{fin} holds for non semiperfect coalgebras seems interesting.
\end{section}

\end{document}